\documentclass[12pt,reqno]{amsart}

\usepackage{amsmath, amssymb, amsfonts, amsbsy, amsthm, latexsym, epsfig, bbm, amscd,graphicx}
\usepackage{dirtytalk}
\usepackage{comment}
\usepackage{xcolor}
\theoremstyle{definition}
\newtheorem{theorem}{Theorem} [section]
\newtheorem{conjecture}[theorem]{Conjecture}

\theoremstyle{definition}
\newtheorem{corollary}[theorem]{Corollary}
\newtheorem{lemma}[theorem]{Lemma}
\newtheorem{proposition}[theorem]{Proposition}
\newtheorem{definition}[theorem]{Definition}

\newtheorem*{namedtheorem}{Theorem}
\newtheorem*{namedlemma}{Lemma}

\numberwithin{equation}{section}
\definecolor{vbcolor}{rgb}{0.35, 0.15, 0.6}

\begin{document}
\title{Submodules of $H^2(\mathbb{T}^2)$ and frames by pairs of bounded commuting operators}

\author{Victor Bailey and Carlos Cabrelli}

\address{Department of Mathematics,
University of Oklahoma,
Norman, OK 73019 USA}
\email{victor.bailey@ou.edu}

\address{Departmentamento de Matemática, Universidad de Buenos Aires, Instituto de Matemática ``Luis Santaló'' (IMAS-CONICET-UBA), Buenos Aires, Argentina}
\email{cabrelli@dm.uba.ar}

\begin{abstract}
Recent work in Dynamical Sampling has been centered on characterizing frames  obtained by the orbit of a vector under a bounded operator. We prove a necessary and sufficient condition for a pair of bounded commuting operators on a separable infinite-dimensional Hilbert space to generate a frame by unilateral iterations on a single vector. Applying the theory on submodules of the Hardy module $H^2(\mathbb{T}^2)$, we characterize these frames in terms of their relation to the two-variable Jordan block on a certain quotient module and provide some properties of frames of this form. 
\end{abstract}
\maketitle
\section
{Introduction}

\quad A central problem in Dynamical Sampling  \cite{ACMT17} asks for the conditions needed to obtain frames and/or bases for a separable Hilbert space $H$ from collections of vectors of the form $\{A^n g\}_{{0 \leq n \leq \Psi (g)} , \, g \, \in \, G}$ 
where $A \in B(H)$, $G$ is a countable subset of $H,$ and $\Psi$ is a map from $G$ to $\mathbb{N}_0 \cup \{\infty\}$ (where $\mathbb{N}_0 = \mathbb{N} \cup \{0\}$).  As a complete solution to this Dynamical Sampling problem in finite-dimensional Hilbert spaces was provided in \cite{ACMT17}, recent works consider the problem where $H$ is a separable infinite-dimensional Hilbert space. In this setting, typically the problem is posed in the following way: Given an operator $T \in B(H)$ and a vector $\varphi \in H$, what are conditions for the system $\{T^n \varphi\}_{n \geq 0}$ to be a frame or have properties that are useful for reconstructing $H$ such as completeness, exactness, etc.

 \quad A necessary and sufficient condition for a system $\{T^n \varphi\}_{n \geq 0} \subset H$ to be a frame was given in \cite{CHP20}. This result exhibited the inextricable connection between Dynamical Sampling and Hardy Space theory as the characterization of the bounded operators that can be used to generate a frame for $H$ hinges on the relationship between the iterating operator and a compression of the shift operator to a certain type of subspace (known as a model space) of the Hardy Space, \begin{displaymath}
H^2(\mathbb{T}) = \{f \in L^2(\mathbb{T}) : \int_{\mathbb{T}} f(z) \overline{z}^n dz = 0 \, \,  \forall \, n<0 \}.
 \end{displaymath}
 In Theorem 3.6 of \cite{CHP20} it is shown that for any complex separable infinite-dimensional Hilbert Space $H$, the collection of vectors $\{T^{n}\varphi\}_{n \geq 0} \subset H$ is a frame for $H$ if and only if for some inner function $\theta \in H^2(\mathbb{T})$ there is an infinite-dimensional model space $K_\theta =  H^2(\mathbb{T}) \ominus \theta H^2(\mathbb{T})$ for which the operator $T$ is similar to the compressed shift operator $S_{\theta} = P_{K_{\theta}} \left.S\right|_{K_{\theta}}$ (where $P_{K_{\theta}}$ the orthogonal projection onto $K_{\theta}$) and $\varphi$ is the image of the function $P_{K_{\theta}}1_{\mathbb{T}}$ under the invertible map given in the similarity relation between $T$ and $S_{\theta}$. 


Recall that a system $\{f_n\}_{n \in I} \subset H$ is said to be a frame if there exist fixed constants $ 0 < C_{1} \leq C_{2}$ such that for each $ f \in H$,  
\begin{displaymath} C_{1} \|f\|^2 \leq \sum_{n \in I} | \langle f , f_n \rangle |^2 \leq C_{2} \|f\|^2 . \quad 
\end{displaymath}
A frame for $H$ admits basis-like expansions for each element of $H$. That is, for a system $\{f_n\}_{n \in I} \subset H $, that satisfies the inequality above, each vector $f\in H$ can be given a series representation of the form  $f = \underset{n \in I}{\sum} c_n \, f_n $
where $\{c_n\}_{n \in I} \in \ell^2(I)$ and convergence is in the norm of $H$.  For a frame, these expansions are not necessarily unique so that a frame can be viewed as a generalized basis for $H$. Frames for which these expansions are not unique, so that the frame is not minimal, are known as overcomplete frames and a minimal frame is known as a Riesz basis. 


 In the sequel, we show that for frames of the form $\{T_{1}^iT_{2}^j\varphi\}_{i,j \geq 0}$ (where $T_{1}, T_{2} \in B(H)$ commute) an analogous relation holds as the one given in Theorem 3.6 of \cite{CHP20}. That is, the bounded and commuting iterating operators $T_{1}$ and $T_{2}$ must be similar to compressions of the two shift operators defined on a certain subspace of the Hardy Space on the bidimensional torus, $H^2(\mathbb{T}^2)$. 
In \cite{ACCP22}, the authors investigated a mixed setting in which one of the operators is applied bilaterally, while the other is applied only unilaterally. This asymmetric structure requires the development of specific techniques tailored to that framework. These methods are substantially different from those employed in our symmetric setting, which introduces its own significant analytical challenges, most notably, the need to work within the Hardy space on the bidisc.
In the space 
\begin{displaymath}
H^2(\mathbb{T}^2) = \{f(z, w) \in L^2(\mathbb{T}^2): \int_{\mathbb{T}^2} f(z,w) \overline{z}^m \overline{w}^n d \mu = 0 \, \, \text{if} \, \, \,  m<0 \, \, \text{or} \, \, n<0 \} 
\end{displaymath}
where $\mu$ is the Haar measure on $\mathbb{T}^2$, many of the characteristics of the space $H^2(\mathbb{T})$ are reflected. On $H^2(\mathbb{T}^2)$ we have two shift operators defined by multiplication by $z$ and multiplication by $w$ respectively and say a subspace $M \subseteq H^2(\mathbb{T}^2)$  is shift-invariant if it is invariant under both shift operators. However, the subspaces corresponding to the model space $K_\theta$ in this setting are not necessarily of the form $H^2(\mathbb{T}^2) \ominus \phi H^2(\mathbb{T}^2)$ as Beurling's characterization of shift-invariant subspaces of $H^2(\mathbb{T})$ does not carry over fully to this setting. For this reason, in order to seamlessly extend the result from Theorem 3.6 of \cite{CHP20} to the operator pair case, we make use of the theory on the structure of shift-invariant subspaces of multidimensional Hardy Spaces which enables us to obtain Corollary \ref{B-type} which follows from the  complete characterization of frames of the form $\{T_{1}^iT_{2}^j\varphi\}_{i,j \geq 0}$ where $T_{1}, T_{2} \in B(H)$ commute that we provide in Theorem \ref{char} below.
\section{Shift-Invariant Subspaces of $H^2(\mathbb{T}^2)$}
Note that the Hardy space has an equivalent definition as a collection of functions defined on the open unit disc $\mathbb{D}$. That is, the space \begin{displaymath}
    H^2(\mathbb{D}) = \{ f: \mathbb{D} \to \mathbb{C} \, \, \text{such that} \, \, f(z) = \underset{n \geq 0}{\sum} c_{n}z^n \, \, \text{where} \, \, \, (c_{n}) \in \ell^2(\mathbb{N}_0)\},
\end{displaymath} 
of analytic functions on $\mathbb{D}$ is known as the Hardy space on the disc and we identify this space of functions with those in $H^2(\mathbb{T})$ via radial limits \cite{MR07}. 
We say a subspace $M \subseteq H^2(\mathbb{T})$  is shift-invariant if it is invariant under the multiplication operator $S \in B(H^2(\mathbb{T}))$ where $Sf(z) = zf(z)$ for all $f \in H^2(\mathbb{T})$.

Similarly, we can identify functions in 
\begin{displaymath}
    H^2(\mathbb{D}^2) = \{ f: \mathbb{D}^2 \to \mathbb{C} \, \, \text{such that} \, \, f(z, w) = \underset{n \geq 0}{\sum} c_{n,m}z^nw^m \, \, \text{where} \, \, \, (c_{n, m}) \in \ell^2(\mathbb{N}_0\times \mathbb{N}_0)\},
\end{displaymath} 
the Hardy Space on the bidisc with those in $H^2(\mathbb{T}^2)$ via radial limits in like fashion. 

 We say a subspace $M \subseteq H^2(\mathbb{T}^2)$  is shift-invariant if it is invariant under both shifts and such subspaces are also referred to in the literature as submodules of the Hardy module $H^2(\mathbb{T}^2)$ \cite{CG03,S15,Y18}.

\begin{definition}
A subspace $M \subseteq H^2(\mathbb{T}^2)$  is shift-invariant, or a submodule of $H^2(\mathbb{T}^2)$, if it is invariant under both shift operators $S_z$ and $S_w$. That is, $M$ is shift-invariant if $S_zM = zM \subset M$ and $S_wM = wM \subset M$.
\end{definition}
 \begin{definition}
     A quotient module is a subspace in $H^2(\mathbb{T}^2)$ of the form $K = H^2(\mathbb{T}^2) \ominus M$ for some submodule $M \subset H^2(\mathbb{T}^2)$.
 \end{definition}
For a given submodule $M$ we associate the pair $(S_{K_z}, S_{K_w})$ to the quotient module $K= H^2(\mathbb{T}^2) \ominus M$ which are each compressions of the shift operators $S_z$ and $S_w$ to the quotient module. In the literature, this pair is known as a two-variable Jordan block \cite{LYY11}.
 \begin{definition}
For a given quotient module $K = H^2(\mathbb{T}^2) \ominus M$ the pair $(S_{K_z}, S_{K_w})$ where $S_{K_z} = P_{K} \left.S_{z}\right|_{K} $ and $S_{K_w} = P_{K} \left.S_{w}\right|_{K}$ (where $P_{K}$ is the orthogonal projection onto $K$) is the two-variable Jordan block on the quotient module $K$.
\end{definition}
Over the past several decades \cite{ACD86,CG03,R69,Y18} the structure of shift-invariant subspaces in $H^2(\mathbb{T}^2)$ ( or in $H^2(\mathbb{D}^2)$) has been the subject of extensive study in the area of multivariable operator theory. This has led to the formulation of the viewpoint of such subspaces as Hilbert modules \cite{CG03,DP89,S15} as subspaces closed under multiplication by $z \, \, \text{and} \, \, w$ are in fact closed under multiplication by all functions in the bidisc algebra $A(\mathbb{D}^2)$ (which is the closure of the polynomials in $C(\overline{\mathbb{D}}^2)$), and to the construction of various tools which can be used to ascertain the properties (such as the rank, codimension, etc.) of a given submodule \cite{Y18}. 
\newline
In the single variable case, that is in the space $H^2(\mathbb{T})$, there is a complete characterization of the structure of shift-invariant subspaces due to Beurling in terms of the inner functions which parametrize the spaces.  
  \begin{definition}
      An inner function in $H^2(\mathbb{D})$ is a function $\theta$ that satisfies $| \theta (z)| = 1$ almost everywhere on $\mathbb{T}$.
\end{definition}
Inner functions in $H^2(\mathbb{D})$ (or $H^2(\mathbb{T})$) are the class of unimodular functions in $H^2(\mathbb{D})$.
\begin{definition}
    A shift-invariant subspace $M$ is said to be unitarily equivalent to $N$ if there exists a unimodular function $\psi \in L^\infty$ such that $M =\psi N $.
\end{definition}
 \quad \, The following theorem characterizing the structure of the shift-invariant subspaces in $H^2(\mathbb{T})$ by Beurling \cite{B48} shows that they  are all unitarily equivalent to the full space $H^2(\mathbb{T})$.
 
   \begin{theorem}[Beurling]
    Every nontrivial invariant subspace of the shift operator $S \in B(H^2(\mathbb{T})) $, where $Sf(z) = zf(z)$, is of the form $\theta H^2(\mathbb{T})$ for some inner function $\theta$. Conversely, any  subspace of the form $\theta H^2(\mathbb{T})$ for some inner function $\theta$ is shift-invariant.
    \end{theorem}

There are several known examples of shift-invariant subspaces in $H^2(\mathbb{T}^2)$ that are not of the form given in Beurling's Theorem. That is, there exist submodules in $H^2(\mathbb{T}^2)$ that are not unitarily equivalent to $H^2(\mathbb{T}^2)$. For this reason numerous works have been aimed at providing characterizations for the structure of submodules \cite{Y18}.

\begin{definition}
  An inner function in $H^2(\mathbb{T}^2)$ is a function $\phi$ that satisfies $| \phi (z, w)| = 1$ almost everywhere on $\mathbb{T}^2$.
\end{definition}

As noted above, Beurling's characterization of shift-invariant subspaces does not translate fully to the multidimensional Hardy Space setting as there exist shift-invariant subspaces in $H^2(\mathbb{T}^2)$ that are not of Beurling-type \cite{R69}, that is of the form $\phi H^2(\mathbb{T}^2)$ for some inner function $\phi \in H^2(\mathbb{T}^2)$.  However, due to Mandrekar \cite{M88}, we have a characterization of the Beurling-type shift-invariant subspaces in $H^2(\mathbb{T}^2)$.

\begin{definition}
A pair of isometries $V_1$ and $V_2$ doubly commute on a Hilbert space, $H$, if $V_1$ commutes with $V_2$ and $V_1$ commutes with $V_2^{*}$.
\end{definition}

\begin{theorem}[Mandrekar]
A nontrivial shift-invariant subspace $M \subset H^2(\mathbb{T}^2)$ is of the form $\phi H^2(\mathbb{T}^2)$ with $\phi(z,w)$ inner if and only if $S_z$ and $S_w$ doubly commute on $M$.
 \end{theorem}

It is important to note that a nontrivial Beurling-type shift-invariant subspace, $\phi H^2(\mathbb{T}^2)$,  always has the property that the associated quotient module is infinite dimensional (that is, $dim(H^2(\mathbb{T}^2) \ominus \phi H^2(\mathbb{T}^2)) = \infty$). In Theorem 3.6 in \cite{CHP20} it was required that the subspace $K_\theta = H^2(\mathbb{T}) \ominus \theta H^2(\mathbb{T})$ had infinite dimension (which equivalently says $\theta H^2(\mathbb{T})$ has infinite codimension), since the theorem necessitates that the inner function $\theta$ not be a finite Blaschke product (which by Theorem 3.14 \cite{RR03} means $\theta H^2(\mathbb{T})$ has infinite codimension). 
Thus, the value of a Beurling-type submodule in $H^2(\mathbb{T}^2)$, for our purposes, is exhibited by the fact that we do not need to impose any conditions on the inner function in this setting to obtain $\phi H^2(\mathbb{T}^2)$ has infinite codimension. That is, we automatically satisfy the condition that the quotient module corresponding to $K_\theta$ in our setting has infinite dimension with any proper Beurling-type shift-invariant subspace. 

\begin{proposition}\label{infcodim}
Every proper submodule of the form $\phi H^2(\mathbb{T}^2)$, where $\phi(z, w)$ is a (nonconstant) inner function, satisfies the property that $dim(H^2(\mathbb{T}^2) \ominus \phi H^2(\mathbb{T}^2)) = \infty$. That is, every proper Beurling-type shift-invariant subspace of $H^2(\mathbb{T}^2)$ has infinite codimension.
\end{proposition}
\begin{proof}
By \cite{ACD86}, a shift-invariant subspace $M \subseteq H^2(\mathbb{T}^2)$ with finite codimension has full range (definition on pg 5 of \cite{ACD86}). By \cite{M88}, shift-invariant subspaces of the form $\phi H^2(\mathbb{T}^2)$, where $\phi(z, w)\in H^2(\mathbb{T}^2)$ is an inner function, do not have full range unless $\phi$ is constant. Therefore, a Beurling-type shift-invariant subspace, $\phi H^2(\mathbb{T}^2)$, cannot have finite codimension unless $\phi$ is constant. That is, shift-invariant subspaces of the form $\phi H^2(\mathbb{T}^2)$  have infinite codimension unless $\phi H^2(\mathbb{T}^2) = H^2(\mathbb{T}^2)$ (since $\phi H^2(\mathbb{T}^2) = H^2(\mathbb{T}^2)$ if and only if $\phi$ is constant).
\end{proof}
A characterization of the submodules $M$ of finite codimension was provided by Ahern and Clark in terms of the topological properties of the zero set of the polynomials in $M$ \cite{Y18}.

\section{Frames by Unilateral Iterations of bounded Commuting Operators}
In what follows, we characterize the pairs of bounded commuting operators which can be used to generate frames for a separable infinite-dimensional Hilbert space. Critical for the characterization provided in Theorem \ref{char} is the following notion of similarity.
 \begin{definition} \label{similar}
Let $H, K$ be complex separable infinite-dimensional Hilbert Spaces. Let $T_1$ and $ T_2$ be commuting operators  in $B(H)$ and let  $V_1$ and $ V_2$ be commuting operators in $B(K)$. We say the triples $(T_1, T_2, \varphi)$ and $(V_1, V_2, f)$ are similar and write 
\newline
$(T_1, T_2, \varphi) \cong (V_1, V_2, f)$ if there exists an invertible map $L \in B(H, K)$ such that $LT_{1}L^{-1} = V_{1}$, $LT_{2}L^{-1} = V_{2}$, and $L\varphi = f$ where $\varphi \in H$ and $f \in K$.
\end{definition}
Note that the similarity relation in the above definition implies that $LT_{1}^iT_{2}^jL^{-1} = V_{1}^iV_{2}^j$ for all $i, j \geq 0$.



\begin{lemma} \label{similarity}
Suppose $(T_1, T_2, \varphi) \cong (V_1, V_2, f)$. Then $\{T_{1}^iT_{2}^j\varphi\}_{i,j \geq 0}$ is a frame if and only if $\{V_{1}^iV_{2}^jf\}_{i,j \geq 0}$ is a frame. In the affirmative case,  the operator $L$ in Definition \ref{similar} is unique. 
\end{lemma}
\begin{proof}
Since $\underset{i,j \geq 0}{\sum} |\langle V_{1}^iV_{2}^jf, g \rangle |^2 = \underset{i,j \geq 0}{\sum} |\langle LT_{1}^iT_{2}^jL^{-1} L\varphi , g\rangle |^2 = \underset{i,j \geq 0}{\sum} |\langle T_{1}^iT_{2}^j\varphi, L^*g \rangle |^2$, the statement follows since $L$ is a topological isomorphism so that $\{T_{1}^iT_{2}^j\varphi\}_{i,j \geq 0}$ being a frame implies $\{LT_{1}^iT_{2}^j\varphi\}_{i,j \geq 0}$ is a frame and thus $\{V_{1}^iV_{2}^jf\}_{i,j \geq 0}$ is a frame by the equation above. Also, assuming $\{V_{1}^iV_{2}^jf\}_{i,j \geq 0}$ is a frame implies $\{LT_{1}^iT_{2}^j\varphi\}_{i,j \geq 0}$ is a frame so that $\{T_{1}^iT_{2}^j\varphi\}_{i,j \geq 0} = \{L^{-1}LT_{1}^iT_{2}^j\varphi\}_{i,j \geq 0}$ is a frame as well. Uniqueness of $L$ holds as for any $f \in H$, we have $f = \underset{i,j \geq 0}{\sum} c_{i,j} T_1^iT_2^j\varphi$ so that 

\begin{displaymath}  
Lf = L \underset{i,j \geq 0}{\sum} c_{i,j} T_1^iT_2^j\varphi = \underset{i,j \geq 0}{\sum} c_{i,j} LT_1^iT_2^jL^{-1} L\varphi
= \underset{i,j \geq 0}{\sum} c_{i,j} V_1^iV_2^jf.
\end{displaymath}

\end{proof}



 The following lemma shows that the kernel of the synthesis operator for a frame generated by unilateral iterations of bounded commuting operators on a fixed vector must be an joint invariant subspace of the two right shift operators on $\ell^{2}(\mathbb{N}_0 \times \mathbb{N}_0)$.
\begin{lemma}
Suppose $\{T_1^iT_2^j\varphi\}_{i,j \geq 0}$ is a frame for H where $T_1, T_2  \in B(H)$ commute. Let $U: \ell^{2}(\mathbb{N}_0 \times \mathbb{N}_0) \to H$ be the synthesis operator of the frame. Let $R_1$ be the right shift in the first component, $R_2$ the right shift in the second component. Then $Ker(U)$ is invariant under $R_1$ and $R_2$.
\end{lemma}
\begin{proof}
Let $\{c_{i, j}\}_{i, j \geq 0} \in Ker(U)$. Then $U \{c_{i, j}\}_{i, j \geq 0} =  \underset{i,j \geq 0}{\sum} c_{i,j} T_1^iT_2^j\varphi = 0$.
\newline 
Note that
\begin{displaymath}
U (R_1 \{c_{i, j}\}_{i, j \geq 0}) =  \underset{i,j \geq 0}{\sum} c_{i,j} T_1^{i+1}T_2^j\varphi =  \underset{j \geq 0}{\sum} c_{0,j} T_1T_2^j\varphi + \underset{j \geq 0}{\sum} c_{1,j} T_1^{2}T_2^j\varphi + \ldots
\end{displaymath}
\begin{displaymath}
= T_1(  \underset{j \geq 0}{\sum} c_{0,j} T_2^j\varphi + \underset{j \geq 0}{\sum} c_{1,j} T_1T_2^j\varphi + \ldots) = T_1( \underset{i,j \geq 0}{\sum} c_{i,j} T_1^iT_2^j\varphi) = T_1(0) = 0.
\end{displaymath}
Thus $Ker(U)$ is invariant under $R_1$. Similarly, $Ker(U)$ is invariant under $R_2$
\end{proof}

Note that by \cite{R85}, the functions in $H^2(\mathbb{T}^2)$ are the radial limits of functions in $H^2(\mathbb{D}^2)$. By \cite{M88,R69} every $f \in H^2(\mathbb{D}^2)$ has the form $f(z_1, z_2) = \underset{i,j \geq 0}{\sum} c_{ij} z_{1}^{i}z_{2}^{j}$, where $i, j \in \mathbb{N}_0$ and $z_1, z_2 \in \mathbb{D}$  and the scalar sequence $\{c_{ij}\}_{i, j \geq 0}$ satisfies  $\underset{i,j \geq 0}{\sum}|c_{ij}|^2 < \infty$. Also, by \cite{M88} given $f \in H^2(\mathbb{D}^2)$ such that $f(z_1, z_2) = \underset{i,j \geq 0}{\sum} c_{ij} z_{1}^{i}z_{2}^{j}$, the radial limits of this function exist almost everywhere so that we can identify $f$ with $\tilde{f}(t_1, t_2) = \underset{i,j \geq 0}{\sum} c_{ij} t_{1}^{i}t_{2}^{j} \in H^2(\mathbb{T}^2)$, where $t_1, t_2 \in \mathbb{T}$. Thus all elements of $H^2(\mathbb{T}^2)$ can be written in the form $\underset{i,j \geq 0}{\sum} c_{ij} z^{i}w^{j}$ with $\{c_{ij}\}_{i, j \geq 0} \in \ell^2(\mathbb{N}_0 \times \mathbb{N}_0)$ and $z, w \in \mathbb{T}$.

Let $A: \ell^2(\mathbb{N}_0 \times \mathbb{N}_0) \to H^2(\mathbb{T}^2)$ where $A(\{c_{ij}\}_{i, j \geq 0}) = \underset{i,j \geq 0}{\sum} c_{ij} z^{i}w^{j}$. Then $A$ maps $\ell^2(\mathbb{N}_0 \times \mathbb{N}_0)$ onto $H^2(\mathbb{T}^2)$ and as for any $ f = \underset{i,j \geq 0}{\sum} a_{ij} z^{i}w^{j}  \in H^2(\mathbb{T}^2)$ we have 
\newline
$\|f\|_{H^2(\mathbb{T}^2)} = \|\{a_{ij}\}_{i, j \geq 0}\|_{\ell^2(\mathbb{N}_0 \times \mathbb{N}_0)}$, this map is unitary. 

Let $R_1$ and $R_2$ be the right shifts on $\ell^2(\mathbb{N}_0 \times \mathbb{N}_0)$ in the first and second components respectively. One can show that 
\newline
$AR_1 \{c_{ij}\}_{i, j \geq 0} = S_z A \{c_{ij}\}_{i, j \geq 0}$ and $AR_2 \{c_{ij}\}_{i, j \geq 0} = S_w A \{c_{ij}\}_{i, j \geq 0}$ so that $AR_1A^{-1} = S_z$ and $AR_2A^{-1} = S_w$. Also, as $A$ is unitary we have $AR_1^{*}A^{-1} = S_z^{*}$ and $AR_2^{*}A^{-1} = S_w^{*}$ as well.

 Applying the observations above, one can show that the images of joint invariant subspaces for $R_1$ and $R_2$ of $\ell^{2}(\mathbb{N}_0 \times \mathbb{N}_0)$ under the operator $A$ are exactly the submodules of $H^2(\mathbb{T}^2)$. Using the fact that the image under $A$ of the  kernel of the synthesis operator for a frame  of the form $\{T_{1}^iT_{2}^j\varphi\}_{i,j \geq 0}$ is then a submodule in $H^2(\mathbb{T}^2)$, assuming the right shift operators doubly commute on this space implies that the corresponding submodule is of Beurling-type.
\begin{lemma}\label{beur}
Let $\{T_{1}^iT_{2}^j\varphi\}_{i,j \geq 0}$ be an overcomplete frame for $H$ where $T_1, T_2 \in B(H)$ commute. Let $U: \ell^2(\mathbb{N}_0 \times \mathbb{N}_0) \to H$ be the synthesis operator for the frame. Assume that $R_1, R_2$ doubly commute on $Ker(U)$. Then the submodule $Ker(V) = Ker(U \mathcal{F}) \subset H^2(\mathbb{T}^2)$ is of the form $\phi H^2(\mathbb{T}^2)$, where $\phi(z,w)$ is an inner function. 
\end{lemma}
\begin{proof}
Suppose $R_1, R_2$ doubly commute on $Ker(U)$. As $S_z$ and $S_w$ commute on $H^2(\mathbb{T}^2)$, we want to show $S_z$ and $S_w^{*}$ commute on the invariant subspace $A(Ker(U))$. 
\newline
Given $R_1 R_2^{*} = R_2^{*} R_1$ on $Ker(U)$, we have on $A(Ker(U))$, 
\begin{displaymath}
S_zS_w^{*}= AR_1 A^{-1} A R_2^{*}A^{-1} =AR_1 R_2^{*}A^{-1} = AR_2^{*} R_1A^{-1} = AR_2^{*}A^{-1} A R_1A^{-1} =  S_w^{*} S_z.
\end{displaymath} Also, it holds that $A(Ker(U)) = Ker(U \mathcal{F})$ and therefore by \cite{M88}, $Ker(U \mathcal{F}) = \phi H^2(\mathbb{T}^2)$ for some inner function $\phi(z,w)$.
\end{proof}

The following fact is well-known; however, we provide it here for clarity. 
\begin{lemma} \label{pars}

The set $\{S_{z}^n S_{w}^m {1}_{\mathbb{T}^2}\}_{m,n \geq 0} = \{z^nw^m\}_{m,n \geq 0}$  forms an orthonormal basis for $H^2(\mathbb{T}^2)$.
\end{lemma}
\begin{lemma} \label{twovar}
Let $M \subset H^2(\mathbb{T}^2)$ be a submodule with infinite codimension. Set $K = H^2(\mathbb{T}^2) \ominus M$ and $S_{K_z} = P_{K} \left.S_{z}\right|_{K} $ and  $S_{K_w} = P_{K} \left.S_{w}\right|_{K}$. 
Then $S_{K_z}^m S_{K_w}^n P_{K} {1}_{\mathbb{T}^2} = P_{K} z^mw^n$ for all $m, n \geq 0$.
\end{lemma}
\begin{proof}
Let $f \in H^2(\mathbb{T}^2)$. Then  $P_{K} S_{z}^m S_{w}^n f = P_{K} S_{z}^m S_{w}^n P_{K} f + P_{K} S_{z}^m S_{w}^n P_{M} f = P_{K} S_{z}^m S_{w}^n P_{K} f$. So that for all $m, n \geq 0$ we have $P_{K} S_{z}^m S_{w}^n =P_{K} S_{z}^m S_{w}^n P_{K}$.
\newline
By the line above also, $P_{K} S_{z}^m  = P_{K} S_{z}^m P_{K} $ and $P_{K} S_{w}^n  = P_{K} S_{w}^n P_{K} $ from which it follows that on $K$, $P_{K} S_{z}^m P_{K} = S_{K_z}^m$ and $P_{K} S_{w}^n P_{K} = S_{K_w}^n$.

In particular, 
\begin{displaymath}
P_{K} z^m w^n =P_{K} S_{z}^m S_{w}^n {1}_{\mathbb{T}^2} = P_{K} S_{z}^m S_{w}^n P_{K} {1}_{\mathbb{T}^2} =
P_{K} S_{z}^m P_{K} (S_{w}^n P_{K} {1}_{\mathbb{T}^2})
\end{displaymath}
\begin{displaymath}
 = P_{K} S_{z}^m P_{K} P_{K} (S_{w}^n P_{K} {1}_{\mathbb{T}^2}) = S_{K_z}^m(P_{K} S_{w}^n P_{K} {1}_{\mathbb{T}^2}) =  S_{K_z}^m S_{K_w}^n P_{K} {1}_{\mathbb{T}^2}.
\end{displaymath}
\end{proof}

The following result shows that any frame given by forward iterations of a pair of bounded commuting operators on a single vector has the property that the operator pair must be similar to the two-variable Jordan block, that is the the pair  $(S_{K_z}, S_{K_w})$, on an infinite-dimensional quotient module $K$. 

\begin{theorem}\label{char}
Let $\{T_1^iT_2^j\varphi\}_{i,j \geq 0} \subset H$, where $T_1, T_2  \in B(H)$ commute. Then $\{T_1^iT_2^j\varphi\}_{i,j \geq 0}$ is an overcomplete frame if and only if there exists a nontrivial submodule $M \subset H^2(\mathbb{T}^2)$ with $dim(H^2(\mathbb{T}^2) \ominus M) = \infty$ such that $(T_1, T_2, \varphi)\cong (S_{K_z}, S_{K_w}, P_{K} {1}_{\mathbb{T}^2})$ where $P_{K}$ is the orthogonal projection onto the quotient module $K = H^2(\mathbb{T}^2) \ominus M$.  
\end{theorem}

\begin{proof}
Assume there exists a nontrivial invariant subspace $M \subset H^2(\mathbb{T}^2)$ with 
\newline
$dim(H^2(\mathbb{T}^2) \ominus M) = \infty$ such that $(T_1, T_2, \varphi)\cong (S_{K_z} S_{K_w}, P_{K}{1}_{\mathbb{T}^2})$ where $P_{K}$ is the orthogonal projection onto $K = H^2(\mathbb{T}^2) \ominus M$ and $S_{K_z} = P_{K} \left.S_{z}\right|_{K} $ and  $S_{K_w} = P_{K} \left.S_{w}\right|_{K}$. By Lemma \ref{pars} and Lemma \ref{twovar}, it follows that the system $\{S_{K_z}^m S_{K_w}^n P_{K} {1}_{\mathbb{T}^2}\}_{m, n \geq 0} = \{P_{K} z^mw^n\}_{m, n \geq 0}$  is an overcomplete Parseval frame for $K$. Hence, by 
Lemma \ref{similarity}, we have  $\{T_1^iT_2^j\varphi\}_{i,j \geq 0}$ is an overcomplete frame for $H$. Conversely, suppose $\{T_1^iT_2^j\varphi\}_{i,j \geq 0}$ is an overcomplete frame for $H$. Define $V = U \mathcal{F}$. Then $V: H^2(\mathbb{T}^2) \to H$ is surjective so that $Ker(V)$ has infinite codimension. That is, $K = H^2(\mathbb{T}^2) \ominus Ker(V)$ satisfies $dim(K) =\infty$. Note that $V {1}_{\mathbb{T}^2} = \varphi$ since $\mathcal{F} {1}_{\mathbb{T}^2} $ gives the scalar sequence in $\ell^2(\mathbb{N}_0 \times \mathbb{N}_0) $ that has the value $1$ in the $(0, 0)$ component and the value $0$ everywhere else. Thus $ V {1}_{\mathbb{T}^2} = U \mathcal{F} {1}_{\mathbb{T}^2} = \varphi$. Similarly for each pair $(i, j)$, $VS_{z}^iS_{w}^j {1}_{\mathbb{T}^2} = Vz^iw^j = T_{1}^iT_{2}^j \varphi$. Set $W = \left.V\right|_{K}$. Then $\varphi = V {1}_{\mathbb{T}^2} = V(P_{K} {1}_{\mathbb{T}^2} + P_{Ker(V)} {1}_{\mathbb{T}^2}) = V P_{K} {1}_{\mathbb{T}^2} = W P_{K} {1}_{\mathbb{T}^2}$ and for every $i, j \geq 0$, $T_{1}^iT_{2}^j \varphi = Vz^iw^j = V P_{K} z^iw^j = W P_{K} z^iw^j= W S_{K_z}^i S_{K_w}^j P_{K} {1}_{\mathbb{T}^2}$. 
\newline
Note that $W \in B(K, H)$ is invertible and as for any $f \in H$, $f = \underset{i,j \geq 0}{\sum} c_{i,j} T_1^iT_2^j\varphi$ so that 
\begin{displaymath}
W S_{K_z}^m S_{K_w}^n W^{-1}f = W S_{K_z}^m S_{K_w}^n W^{-1} \underset{i,j \geq 0}{\sum} c_{i,j} T_1^iT_2^j\varphi = W S_{K_z}^m S_{K_w}^n  \underset{i,j \geq 0}{\sum} c_{i,j} W^{-1}T_1^iT_2^j\varphi 
\end{displaymath}
\begin{displaymath}
=W S_{K_z}^m S_{K_w}^n  \underset{i,j \geq 0}{\sum} c_{i,j} S_{K_z}^i S_{K_w}^j P_{K} {1}_{\mathbb{T}^2} = W \underset{i,j \geq 0}{\sum} c_{i,j} S_{K_z}^{i+m} S_{K_w}^{j+n} P_{K} {1}_{\mathbb{T}^2}  
\end{displaymath}
\begin{displaymath}
=\underset{i,j \geq 0}{\sum} c_{i,j} W S_{K_z}^{i+m} S_{K_w}^{j+n} P_{K} {1}_{\mathbb{T}^2} = \underset{i,j \geq 0}{\sum} c_{i,j} T_1^{i+m}T_2^{j+n}\varphi = T_1^{m}T_2^{n} \underset{i,j \geq 0}{\sum} c_{i,j} T_1^iT_2^j\varphi = T_1^{m}T_2^{n} f.
\end{displaymath}
\end{proof}

The following corollary provides a seamless extension from Theorem 3.6 in \cite{CHP20}. Applying Lemma \ref{beur} we impose a condition which ensures that the quotient module $K$ in Theorem \ref{char} is given by a submodule of Beurling type. This condition enables us to remove the assumption that the quotient module $K$ has infinite dimension in Theorem \ref{char} as by Proposition \ref{infcodim} proper submodules of Beurling-type always have infinite codimension. 
It follows directly from Lemma 1 in \cite{R85} that if for some inner functions $\phi_1$ and $\phi_2$ we have
$\phi_1 H^2(\mathbb{T}^2) = \phi_2 H^2(\mathbb{T}^2)$, then $\phi_1(z,w)/\phi_2(z,w)$ is a constant. 
Hence when 
$\phi_1 H^2(\mathbb{T}^2) = \phi_2 H^2(\mathbb{T}^2)$, the inner functions $\phi_1$ and $\phi_2$ differ by a unimodular constant factor, $\phi_1(z,w)/\phi_2(z,w)$. From this, we have uniqueness of the inner function in the following corollary up to a unimodular constant factor.  
 \begin{corollary}\label{B-type}
Let $\{T_1^iT_2^j\varphi\}_{i,j \geq 0} \subset H$, where $T_1, T_2  \in B(H)$ commute, satisfy the property that the operators $R_1, R_2$ doubly commute on the kernel of the synthesis operator for $\{T_1^iT_2^j\varphi\}_{i,j \geq 0}$. Then $\{T_1^iT_2^j\varphi\}_{i,j \geq 0}$ is an overcomplete frame for $H$ if and only if there exists a proper Beurling-type submodule $\phi H^2(\mathbb{T}^2)$, with $\phi (z,w)$ a unique inner function, such that $(T_1, T_2, \varphi)\cong (S_{K_z}, S_{K_w}, P_{K} {1}_{\mathbb{T}^2})$ where $P_{K}$ is the orthogonal projection onto the quotient module $K = H^2(\mathbb{T}^2) \ominus \phi H^2(\mathbb{T}^2$).    
\end{corollary}
\begin{proof}
This theorem holds as a special case of Theorem \ref{char} where the shift-invariant subspace $M$ is of Beurling-type. By Proposition \ref{beur}, Beurling-type shift-invariant subspaces always have infinite codimension.

\end{proof}

The following corollary provides a characterization of the pairs of bounded commuting operators which can be
used to generate a Riesz basis (or minimal frame) for a separable infinite-dimensional Hilbert space.
\begin{corollary}
Let $\{T_1^iT_2^j\varphi\}_{i,j \geq 0} \subset H$, where $T_1, T_2  \in B(H)$ commute. Then $\{T_1^iT_2^j\varphi\}_{i,j \geq 0}$ is a Riesz basis if and only if $(T_1, T_2, \varphi)\cong (S_z, S_w, {1}_{\mathbb{T}^2})$. 
\end{corollary}

\begin{proof}
Suppose  $(T_1, T_2, \varphi) \cong (S_z, S_w, {1}_{\mathbb{T}^2})$. We know that  $\{S_{z}^n S_{w}^m {1}_{\mathbb{T}^2}\}_{m,n \geq 0} = \{z^nw^m\}_{m,n \geq 0}$ \,  forms an orthonormal basis for $H^2(\mathbb{T}^2)$. It follows that there exists a topological isomorphism, $L$, such that $\{L S_{z}^n S_{w}^m {1}_{\mathbb{T}^2}\}_{n,m \geq 0} = \{T_1^nT_2^m\varphi\}_{n,m \geq 0}$. Since the image of an orthonormal basis under a topological isomorphism is a Riesz basis, we have  $\{T_1^iT_2^j\varphi\}_{i,j \geq 0}$ is a Riesz basis for $H$.
For the other direction, assume that $\{T_1^iT_2^j\varphi\}_{i,j \geq 0}$ is a Riesz basis. Define $V = U \mathcal{F}$. Then as $U$ is bijective, we get $V \in B(H^2(\mathbb{T}^2), H)$ is invertible and the result follows by setting $K = H^2(\mathbb{T}^2)$ in the proof of Theorem \ref{char} above.
\end{proof}
\section{Properties of Frames of Iterations via Commuting Operators}
The following proposition shows that for a frame of the form $\{T_1^iT_2^j\varphi\}_{i,j \geq 0}$ the adjoint of the iterating operator pair $(T^*_1, T^*_2)$ reflects the property as exhibited in the single operator case, that the adjoint of the  iterating operator and goes to zero in the strong operator topology \cite{AA17}.

\begin{proposition}
    Let $\{T_1^iT_2^j\varphi\}_{i,j \geq 0} \subset H$ be a frame for $H$  where $T_1, T_2  \in B(H)$ commute. Then $(T_1^*)^i(T_2^*)^jf \to 0$ as $i, j \to \infty$.
\end{proposition}
\begin{proof}
    Note that 
    \begin{displaymath}
        (T_1^*)^i(T_2^*)^j = (T_2^jT_1^i)^* = (T_1^iT_2^j)^* = (T_2^*)^j(T_1^*)^i \, \, \, \forall i, j \geq 0.
        \end{displaymath}
Given $\{T_1^iT_2^j\varphi\}_{i,j \geq 0} \subset H$ is a frame with frame bounds $A , B >0, $ we have 
\begin{displaymath}
\underset{i, j \geq 0}{\sum}|\langle T_{1}^iT_{2}^j \, \varphi, (T_{1}^{m_1}T_{2}^{m_2})^* f \rangle |^2 =  \underset{i, j \geq 0}{\sum} |\langle T_{1}^{i+m_1}T_{2}^{j + m_2} \, \varphi, f \rangle |^2  =  \underset{j \geq m_2}{\sum} \, \, \underset{i \geq m_1}{\sum}|\langle T_{1}^{i}T_{2}^{j} \, \varphi, f \rangle |^2 \leq B \|f\|^2.
\end{displaymath}
Since the sum $\underset{j \geq m_2}{\sum} \, \, \underset{i \geq m_1}{\sum}|\langle T_{1}^{i}T_{2}^{j} \, \varphi, f \rangle |^2 $ converges, this implies that $|\langle T_{1}^{i}T_{2}^{j} \, \varphi, f \rangle |^2 \to 0$ as $m_1, m_2 \to \infty$. Also, as 
\begin{displaymath}
A \|(T_{1}^{m_1}T_{2}^{m_2})^* f\|^2 \leq \underset{i, j \geq 0}{\sum}|\langle T_{1}^iT_{2}^j \, \varphi, (T_{1}^{m_1}T_{2}^{m_2})^* f \rangle |^2  = \underset{j \geq m_2}{\sum} \, \, \underset{i \geq m_1}{\sum}|\langle T_{1}^{i}T_{2}^{j} \, \varphi, f \rangle |^2 
\end{displaymath}
we have $(T_1^*)^{m_1} (T_2^*)^{m_2}f = (T_{1}^{m_1}T_{2}^{m_2})^* f \to 0$ as as $m_1, m_2 \to \infty$
\end{proof}

In the case where the  frame $\{T_1^iT_2^j\varphi\}_{i,j \geq 0} \subset H$ satisfies the conditions of Lemma \ref{beur} we conjecture that the operator pair $(T_1, T_2)$ also reflects the property, as exhibited in the single operator case, that the iterating operator goes to zero in the strong operator topology \cite{CHP20}.

\begin{conjecture}
 Assume $\{T_1^iT_2^j\varphi\}_{i,j \geq 0}$ is an overcomplete frame for $H$, where $T_1, T_2  \in B(H)$ commute, such that that the operators $R_1, R_2$ doubly commute on the kernel of the synthesis operator  for the frame. Then for each $f \in H$ we have that  $T_1^iT_2^jf \to 0$ as $i, j \to \infty$.
\end{conjecture}
Of interest in the area of Dynamical Frames (i.e. semigroup representation frames)  is determining semigroup representations (admitting frame vectors) for which their frame vectors are equivalent. Namely, when are each of the frame vectors for the frame representation of a semigroup the image of a fixed frame vector under an invertible operator in the commutant of the collection of operators given by the semigroup representation. Note that in the literature such frame representations are known as central frame representations \cite{BHKLL25}.

Theorem 3.9 in \cite{CHP20} shows that the frame vectors are equivalent for any frame representation of the semigroup $Z_+$ (which are precisely the frames of the form $\{T^n f\}_{n\geq 0}$ where $T \in B(H)$). It was recently shown in \cite{BHKLL25} that the frame vectors are also equivalent for any frame representation of the semigroup $Z_+^n$. This result implies the following.

\begin{proposition}
   Let $\{T_1^iT_2^j \varphi\}_{i,j \geq 0} \subset H$, where $T_1, T_2  \in B(H)$ commute be a frame.  Then $\{T_1^iT_2^j f\}_{i,j \geq 0}$ is also a frame for $H$ if and only if $f = V \varphi$ for some invertible $V \in B(H)$ that commutes with both $T_1,$ and $T_2$.
   
\end{proposition}
\begin{proof}
    This follows immediately from Theorem 3.4 in \cite{BHKLL25}.
\end{proof}
\begin{proposition}
   Let $\{T_1^iT_2^j \varphi\}_{i,j \geq 0} \subset H$, where $T_1, T_2 \in B(H)$ commute be an overcomplete frame such that $R_1, R_2$ doubly commute on the kernel of its synthesis operator. Then for any frame $\{g_{i,j}\}_{i,j \geq 0} $  equivalent to $\{T_1^iT_2^j \varphi\}_{i,j \geq 0} \subset H$ (i.e. $ \{V(T_1^iT_2^j \varphi)\}_{i,j \geq 0} =  \{g_{i,j}\}_{i,j \geq 0}$ for some invertible $V \in B(H)$), we have $R_1, R_2$ doubly commute on the kernel of its synthesis operator as well. In particular, if $V$ commutes with both $T_1,$ and $T_2$, then replacing $\varphi$ with $g_{0 , 0}$ we obtain a frame of unilateral iterations satisfying the conditions of Corollary \ref{B-type}.
\end{proposition}

\begin{proof}
    Let $\{g_{i,j}\}_{i,j \geq 0} $ be a frame for $H$ (note that any frame can be indexed in this way). Observe that given the fact that  $\{g_{i,j}\}_{i,j \geq 0} $ is equivalent to $ \{T_1^iT_2^j \varphi\}_{i,j \geq 0}$,  the kernel of their synthesis operators coincide. This implies that $\{g_{i,j}\}_{i,j \geq 0} $ is an overcomplete frame and  $R_1, R_2$ doubly commute on the kernel of its  synthesis operator. The last statement follows as $g_{0 , 0}$ = $V\varphi$ so that $g_{0 , 0}$ is a frame vector for the pair $(T_1, T_2)$ and $\{g_{i,j}\}_{i,j \geq 0} = \{T_1^iT_2^j g_{0 , 0}\}_{i,j \geq 0}$.
\end{proof}

\end{document}